\documentclass[10pt]{amsart}
\usepackage{amsmath,amssymb,latexsym,cite}
\usepackage[small]{caption}
\usepackage{graphicx,wasysym,overpic,tikz,color}
\usepackage{subfigure,color}

\usepackage[colorlinks=true,urlcolor=blue,
citecolor=red,linkcolor=blue,linktocpage,pdfpagelabels,
bookmarksnumbered,bookmarksopen]{hyperref}
\usepackage[english]{babel}
\usepackage{units}
\usepackage{enumitem}

\usepackage[left=3.2cm,right=3.2cm,top=2.45cm,bottom=2.45cm]{geometry}
\usepackage[hyperpageref]{backref}

\usepackage[colorinlistoftodos]{todonotes}

\makeatletter
\providecommand\@dotsep{5}
\def\listtodoname{List of Todos}
\def\listoftodos{\@starttoc{tdo}\listtodoname}
\makeatother

\numberwithin{equation}{section}
\newtheorem{theorem}{Theorem}[section]
\newtheorem{proposition}[theorem]{Proposition}
\newtheorem{lemma}[theorem]{Lemma}

\newtheorem{example}[theorem]{Example}
\newtheorem{corollary}[theorem]{Corollary}

\theoremstyle{definition}

\newcommand{\N}{{\mathbb N}}
\newcommand{\R}{{\mathbb R}}

\newcommand{\eps}{\varepsilon}

\newcommand{\pnorm}[2][]{\if #1'' \left|#2\right|_p \else \left|#2\right|_{#1} \fi}

\title{$\nicefrac{1}{2}$-Laplacian problems with exponential nonlinearity}

\author[A.\ Iannizzotto]{Antonio Iannizzotto}
\author[M.\ Squassina]{Marco Squassina}

\address{Dipartimento di Informatica \newline\indent
Universit\`a degli Studi di Verona
\newline\indent
C\'a Vignal 2, Strada Le Grazie 15, I-37134 Verona, Italy}
\email{marco.squassina@univr.it}
\email{antonio.iannizzotto@univr.it}

\thanks{The second-named author was supported by
2009 national MIUR project: {\em ``Variational and Topological Methods in the Study of Nonlinear Phenomena''}}


\begin{document}

\begin{abstract}
By exploiting a suitable Trudinger-Moser inequality for fractional Sobolev spaces, 
we obtain existence and multiplicity of solutions for a class of one-dimensional nonlocal
equations with fractional diffusion and nonlinearity at exponential growth.
\end{abstract}

\maketitle

\section{Introduction and results}

\noindent
Since the seminal results by Trudinger \cite{T} and Moser \cite{M} on embeddings
of exponential type for the Sobolev spaces $H^1_0(\Omega)$ with $\Omega\subset\R^2$, many contributions have appeared
related to applications of these results to semi-linear elliptic partial differential equations such as
\begin{equation}
\label{prob-loc}
\left\{  \begin{array}{ll}
    -\Delta u=f(u) &\mbox{in $\Omega$} \\
    u=0 &\mbox{on $\partial\Omega$,}
        \end{array}
      \right.
\qquad 
f(t)\sim e^{4\pi t^{q}}\quad\text{as $t\to+\infty$},
\qquad 
0<q\leq 2,
\end{equation}
where the case $q<2$ is considered a subcritical growth, while the case $q=2$
is known as the critical case with respect to the Trudinger-Moser inequality (see \cite[Theorem 1]{M})
\begin{equation}
\label{trudmos}
\sup_{\substack{u\in H^1_0(\Omega) \\ \|\nabla u\|_{L^2(\Omega)}\leq 1}}\int_\Omega e^{4\pi u^2}\leq C|\Omega|.
\end{equation}
For existence and multiplicity of solutions for problems like \eqref{prob-loc} via techniques of critical point theory, we refer the 
reader to De Figueiredo, Miyagaki $\&$ Ruf \cite{DMR} and to the references therein. Many extensions of inequality \eqref{trudmos} have been achieved,
for spaces $W^{1,n}_0(\Omega)$ with $\Omega\subset\R^n$ 
and also to higher order Sobolev spaces (see Adams \cite{adams}). 
As a consequence, quasi-linear problems involving the $n$-Laplacian on domains $\Omega\subset\R^n$ or the 
linear biharmonic operator $\Delta^2$ for functions of $W^{2,2}_0(\Omega)$
on domains $\Omega\subset\R^4$ can be studied. Focusing the attention on nonlinear problems at exponential growth
involving linear diffusion, if the dimension four is natural for the biharmonic operator $\Delta^2$ and
dimension two is natural for the laplacian $-\Delta$, the natural setting for the fractional
diffusion $(-\Delta)^{1/2}$ is dimension one.
\vskip2pt
\noindent
Fractional Sobolev spaces are well known since the beginning of the last century, especially in the framework of harmonic analysis. More recently, after the paper of Caffarelli $\&$ Silvestre \cite{caffarelli},
a large amount of papers were written on problems
involving the fractional diffusion $(-\Delta)^s$, $0<s<1$. Due to its nonlocal nature, working on bounded
domains suggests the functions to be defined on the whole $\R^n$ and that the problems
are formulated as follows (see Servadei $\&$ Valdinoci \cite{SV}):
\begin{equation}
\label{sfrac}
\left\{  \begin{array}{ll}
    (-\Delta)^{s} u=f(u) &\mbox{in $\Omega$} \\
    u=0 &\mbox{in $\R^n\setminus \Omega$.}
        \end{array}
      \right.
\end{equation}
For the functional framework of fractional Sobolev spaces and fractional
Laplacian, we refer the reader to the survey of Di Nezza, Palatucci $\&$ Valdinoci \cite{DPV}.
Equations like \eqref{sfrac} appear
in fractional quantum mechanics in the study of
particles on stochastic fields modeled by L\'evy processes
which occur widely in physics and biology
and recently the stable L\'evy processes
have attracted much interest.
One dimensional cases have been studied 
by Weinstein \cite{weinstein}. 
\vskip2pt
\noindent
In the present paper we will prove the existence and multiplicity of 
solutions for a Dirichlet problem driven by the  $\nicefrac{1}{2}$-Laplacian operator of the following type:
\begin{equation}
\label{problema-1d}
\tag{$P$} \ \ \ \ \ \ \ \ \ 
\left\{  \begin{array}{ll}
    (-\Delta)^{\nicefrac{1}{2}} u=f(u) &\mbox{in $(0,1)$} \\
    u=0 &\mbox{in $\R\setminus (0,1)$,}
        \end{array}
      \right.
\end{equation}
equivalently written in $(0,1)$ as the nonlocal equation
$$
\frac{1}{2\pi}\int_{\R}\frac{u(x+y)+u(x-y)-2 u(x)}{|y|^{2}}dy+f(u)=0.
$$
It is natural to work on the space
\begin{equation}
\label{Xdef}
X=\left\{u\in H^{\nicefrac{1}{2}}(\R): \ u=0 \ \mbox{in $\R\setminus(0,1)$} \right\}, \quad \|u\|_X=[u]_{H^{\nicefrac{1}{2}}(\R)}
\end{equation}
where $[\cdot]_{H^{\nicefrac{1}{2}}(\R)}$ denotes the Gagliardo semi-norm (see Proposition \ref{hilb}). For this space, 
we state (see Corollary \ref{mti}) and exploit the following Trudinger-Moser type inequality: there exists $0<\omega\leq\pi$ such that for all $0<\alpha<2\pi\omega$ we can find $K_\alpha>0$ such that
\begin{equation}
\label{ourTM}
\int_0^1 e^{\alpha u^2}dx\leq K_\alpha,\qquad
\text{for all $u\in X$, $\|u\|_X\leq 1$.}
\end{equation}
\vskip2pt
\noindent
We list below our hypotheses on the nonlinearity $f$ in the \textit{subcritical} case:
\begin{itemize}
\item[{\bf H}] Let $f\in C(\R)$ be a function such that $f(0)=0$ and denote
\[F(t)=\int_0^t f(\tau)d\tau, \quad \mbox{for all $t\in\R$.}\]
Moreover, assume that there exist $t_0,M>0$ such that:
\smallskip
\begin{itemize}
\item[$(i)$] $0<F(t)\leq M|f(t)|,$  \,\,\quad for all $|t|\geq t_0$;
\vskip4pt
\item[$(ii)$] $0<2F(t)\leq f(t)t$, \qquad for all $t\neq 0$;
\vskip4pt
\item[$(iii)$] $\displaystyle\limsup_{t\to 0}\frac{F(t)}{t^2}<\frac{\lambda_1}{4\pi},$ \quad ($\lambda_1$ provided by Proposition \ref{poinc} below);
\vskip4pt
\item[$(iv)$] $\displaystyle\lim_{|t|\to\infty}\frac{|f(t)|}{e^{\alpha t^2}}=0$, \qquad for all $\alpha>0$.
\end{itemize}
\end{itemize}
\vskip2pt
By a (weak) solution of problem \eqref{problema-1d} we mean a function $u\in X$ satisfying \eqref{weaks} (see Section 3). The following are our main results:

\begin{theorem}\label{subcrit}
If ${\bf H}$ hold, then \eqref{problema-1d} has a nontrivial solution $u\in H^{\nicefrac{1}{2}}(\R)$.
If in addition $f$ is odd, then \eqref{problema-1d} has infinitely many solutions in $H^{\nicefrac{1}{2}}(\R)$.
\end{theorem}

\noindent
On symmetric domains, we also have the following result:

\begin{theorem}\label{subcrit-sym}
If ${\bf H}$ hold, then the problem
\begin{equation*} 
\left\{  \begin{array}{ll}
    (-\Delta)^{\nicefrac{1}{2}} u=f(u) &\mbox{in $(-1,1)$} \\
    u=0 &\mbox{in $\R\setminus (-1,1)$,}
        \end{array}
      \right.
\end{equation*}
has an even nontrivial solution $u\in H^{\nicefrac{1}{2}}(\R)$
decreasing on $\R^+$.
\end{theorem}

\noindent
Now we turn to the \textit{critical} case, under the following assumptions:

\begin{itemize}
\item[${\bf H'}$] Assume ${\bf H}(i)-(iii)$ and:
\begin{itemize}
\item[$(iv)$] there exists $0<\alpha_0<2\pi\omega$ such that
\[\lim_{|t|\to\infty}\frac{|f(t)|}{e^{\alpha t^2}}=\left\{
	\begin{array}{ll}
		\infty  & \mbox{if } 0<\alpha<\alpha_0 \\
		0 & \mbox{if } \alpha>\alpha_0
	\end{array};
\right.
\]
\item[$(v)$] there exists $\psi\in X$ such that $\|\psi\|_X=1$ and
\[\sup_{t\in\R^+}\Big(\frac{t^2}{4\pi}-\int_0^1 F(t \psi)dx\Big)<\frac{\omega}{2\alpha_0}.\]
\end{itemize}
\end{itemize}

\noindent
For this case we have the following result:

\begin{theorem}
\label{C-state}
If ${\bf H'}$ hold, then \eqref{problema-1d} has a nontrivial solution $u\in H^{\nicefrac{1}{2}}(\R)$.
\end{theorem}

\noindent
These results establish a one dimensional fractional counterpart (with the additional information
of symmetry and monotonicity of the solution in Theorem~\ref{subcrit-sym} for symmetric domains)
of the results of \cite{DMR} for the local case in dimension two. 
As far as the critical case is concerned, typically when 
$f(t)\sim e^{\alpha_0 t^{2}}$ as $t\to\infty$,
it is still unclear how to detect suitable (concentrating) optimizing sequences in $X$ for the fractional Trudinger-Moser 
inequality \eqref{ourTM}. However, we can prove that in this case the functional associated
to the problem satisfies the Palais-Smale condition at 
each level $c<\omega/(2\alpha_0)$ and that the
problem has a nontrivial solution 
under the additional hypothesis ${\bf H'}(v)$.
We point out that, with a similar machinery, 
existence and multiplicity of solutions for 
fractional non-autonomous problems like
\begin{equation*} 
\left\{  \begin{array}{ll}
    (-\Delta)^{\nicefrac{1}{2}} u=f(x,u) &\mbox{in $(a,b)$} \\
    u=0 &\mbox{in $\R\setminus (a,b)$,}
        \end{array}
      \right.
\end{equation*}
can be obtained under suitable assumptions on $f:(a,b)\times\R\to\R$.
\medskip

\section{Preliminaries}

\noindent
First we recall some basic facts about the $\nicefrac{1}{2}$-Laplacian operator and the related function 
space $H^{\nicefrac{1}{2}}(\R)$, following mainly \cite{DPV}. For all $s\in (0,1)$, all measurable $u$ and all $x\in\R$ we set
\[(-\Delta)^s u(x)=-\frac{C_s}{2}\int_{\R}\frac{u(x+y)+u(x-y)-2 u(x)}{|y|^{1+2s}}dy,\]
with the constant
\[C_s=\Big[\int_\R\frac{1-\cos(\xi)}{|\xi|^{1+2s}}d\xi\Big]^{-1}\]
(see \cite[Lemma 3.3]{DPV}). We focus on the case $s=\nicefrac{1}{2}$. Note that $C_{\nicefrac{1}{2}}=\pi^{-1}$. We define
\[H^{\nicefrac{1}{2}}(\R)=\left\{u\in L^2(\R): \ \int_{\R^2}\frac{(u(x)-u(y))^2}{|x-y|^2}dxdy<\infty\right\}\]
and for all $u\in H^{\nicefrac{1}{2}}(\R)$ we introduce the Gagliardo seminorm
\[[u]_{H^{\nicefrac{1}{2}}(\R)}=\left[\int_{\R^2}\frac{(u(x)-u(y))^2}{|x-y|^2}dxdy\right]^\frac{1}{2}\]
and the norm
\[\|u\|_{H^{\nicefrac{1}{2}}(\R)}=\left(\|u\|^2_{L^2(\R)}+[u]^2_{H^{\nicefrac{1}{2}}(\R)}\right)^\frac{1}{2}.\]
We know that $(H^{\nicefrac{1}{2}}(\R),\|\cdot\|_{H^{\nicefrac{1}{2}}(\R)})$ is a Hilbert space. Morever, by \cite[Proposition 3.6]{DPV} 
\begin{equation}\label{equinorm}
\|(-\Delta)^{\nicefrac{1}{4}}u\|_{L^2(\R)}=(2\pi)^{-\frac{1}{2}}[u]_{H^{\nicefrac{1}{2}}(\R)},\quad
\text{for all $u\in H^{\nicefrac{1}{2}}(\R)$}.
\end{equation}
Our main tool is a fractional Trudinger-Moser inequality (see Ozawa \cite[Theorem 1]{O} and Kozono, Sato $\&$ Wadade \cite[Theorem 1.1]{KSW}):

\begin{theorem}\label{oza}
There exists $0<\omega\leq\pi$ with the following property: for all $0<\alpha<\omega$ there exists $H_\alpha>0$ such that 
\[\int_\R\big(e^{\alpha u^2}-1\big)dx\leq H_\alpha\|u\|_{L^2(\R)}^2,\]
for every $u\in H^{\nicefrac{1}{2}}(\R)$ with $\|(-\Delta)^{\nicefrac{1}{4}}u\|_{L^2(\R)}\leq 1$.
\end{theorem}

\noindent
We do not possess an explicit formula for the optimal constant $\omega$, and neither we know whether the inequality above holds for $\alpha=\omega$.
\vskip2pt
\noindent
Now we turn to the space $X$, defined in \eqref{Xdef}. Clearly the only constant function in $X$ is $0$, so the seminorm $[\cdot]_{H^{\nicefrac{1}{2}}(\R)}$ turns out to be a norm on $X$, which we denote by $\|\cdot\|_X$. We have the following Poincar\'e-type inequality:

\begin{proposition}\label{poinc}
There exists $\lambda_1>0$ such that for all $u\in X$
\[\|u\|_{L^2(0,1)}\leq\lambda_1^{-\frac{1}{2}}\|u\|_X.\]
Moreover, equality is realized by some $u\in X$ with $\|u\|_{L^2(0,1)}=1$.
\end{proposition}
\begin{proof}
We set
\[S=\left\{u\in X: \ \|u\|_{L^2(0,1)}=1\right\}\]
and equivalently prove that
\begin{equation}\label{lam}
\inf_{u\in S}\|u\|_X^2=\lambda_1>0.
\end{equation}
Clearly $\lambda_1\geq 0$. We first prove that $\lambda_1$ is attained in $S$. Let $(u_n)\subset S$ be a minimizing sequence for \eqref{lam}. In particular, $\sup_{n\in\N}[u]^2_{H^{\nicefrac{1}{2}}(\R)}<\infty$ and $(u_n)$ is bounded in $L^2(0,1)$. In light of \cite[Theorem 7.1]{DPV}, there exists $u\in L^2(0,1)$ such that, up to a subsequence, $u_n\to u$ in $L^2(0,1)$. We extend $u$ by setting $u(x)=0$ for all $x\in\R\setminus(0,1)$, so $u\in L^2(\R)$ and $u_n\to u$ a.e. in $\R$. Fatou's lemma yields
\[\int_{\R^2}\frac{|u(x)-u(y)|^2}{|x-y|^2}dx\leq\liminf_{n}\int_{\R^2}\frac{|u_n(x)-u_n(y)|^2}{|x-y|^2}dx=\lambda_1,\]
hence $u\in X$. Moreover, $\|u\|_{L^2(0,1)}=1$, hence $u\in S$, in particular $u\neq 0$ and $\|u\|_X^2=\lambda_1>0$.
\end{proof}

\noindent
Due to Proposition \ref{poinc}, we can prove further properties of $X$:

\begin{proposition}\label{hilb}
$(X,\|\cdot\|_X)$ is a Hilbert space.
\end{proposition}
\begin{proof}
Clearly the norm $\|\cdot\|_X$ is induced by a inner product, defined for all $u,v\in X$ by
\[\langle u,v\rangle_X=\int_{\R^2}\frac{(u(x)-u(y))(v(x)-v(y))}{|x-y|^2}dxdy.\]
Moreover, by Proposition \ref{poinc} we have for all $u\in X$
\begin{equation}\label{x-norm}
\|u\|_X\leq\|u\|_{H^{\nicefrac{1}{2}}(\R)}\leq(\lambda_1^{-1}+1)^{\frac{1}{2}}\|u\|_X.
\end{equation}
So, completeness of $X$ follows at once from that of $H^{\nicefrac{1}{2}}(\R)$.
\end{proof}

\noindent
We specialize Theorem \ref{oza} to the space $X$:

\begin{corollary}\label{mti}
For all $0<\alpha<2\pi\omega$ there exists $K_\alpha>0$ such that
\[\int_0^1 e^{\alpha u^2}dx\leq K_\alpha\]
for all $u\in X$, $\|u\|_X\leq 1$.
\end{corollary}
\begin{proof}
Fix $u\in X$ with $\|u\|_X\leq 1$. Set $v=(2\pi)^{\nicefrac{1}{2}}u$, then $v\in H^{\nicefrac{1}{2}}(\R)$ and by (\ref{equinorm}) we have $\|(-\Delta)^{\nicefrac{1}{4}}v\|_{L^2(\R)}\leq 1$. Set $\tilde\alpha=(2\pi)^{-1}\alpha$, so $0<\tilde\alpha<\omega$ and by Theorem \ref{oza} and Proposition \ref{poinc} we have
\begin{align*}
\int_0^1 e^{\alpha u^2}dx & =\int_\R \big[e^{\tilde\alpha v^2}-1\big]dx+1 \\
& \leq H_{\tilde\alpha}\|v\|_{L^2(0,1)}^2+1\leq\frac{2\pi H_{\tilde\alpha}}{\lambda_1}+1:=K_\alpha,
\end{align*}
which concludes the proof.
\end{proof}

\noindent
We point out a important consequence of the results above:

\begin{proposition}\label{sum}
$e^{u^2}\in L^1(0,1)$ for every $u\in X$.
\end{proposition}
\begin{proof}
We follow Trudinger \cite{T}. Choose $0<\alpha<\omega$ and set for all $t\in\R$
\[
\phi(t)=\frac{e^{\alpha t^2}-1}{H_\alpha} \quad \mbox{($H_\alpha$ defined as in Theorem \ref{oza}).}
\]
We introduce the Orlicz norm induced by $\phi$ putting for all measurable $u:(0,1)\to\R$
\[\|u\|_\phi=\inf\Big\{\gamma>0 \ : \ \int_0^1 \phi\Big(\frac{u}{\gamma}\Big)dx\leq 1\Big\},\]
and the corresponding Orlicz space $L_{\phi^*}(0,1)$, see Krasnosel'ski\u{\i} $\&$ Ruticki\u{\i} \cite[p.67]{KR} for the definition. 
We prove (by identifying a function $v\in X$ with its restriction to $(0,1)$) that
\begin{equation}\label{cont-emb}
X\hookrightarrow L_{\phi^*}(0,1) \quad\mbox{continuously.}
\end{equation}
For all $v\in X\setminus\{0\}$, we set $w=\|v\|_{H^{\nicefrac{1}{2}}(\R)}^{-1}v$, so by (\ref{equinorm})
\[\|(-\Delta)^{\nicefrac{1}{4}}w\|_{L^2(\R)}=\frac{[v]_{H^{\nicefrac{1}{2}}(\R)}}{(2\pi)^{\nicefrac{1}{2}}\|v\|_{H^{\nicefrac{1}{2}}(\R)}}\leq (2\pi)^{-\nicefrac{1}{2}}<1.\]
So, in light of Theorem~\ref{oza}, we have
\[
\int_0^1 \phi\Big(\frac{v}{\|v\|_{H^{\nicefrac{1}{2}}(\R)}}\Big)dx=\int_\R\frac{e^{\alpha w^2}-1}{H_\alpha} dx\leq\|w\|_{L^2(\R)}^2\leq 1,
\]
hence by \eqref{x-norm}
\[\|v\|_\phi\leq\|v\|_{H^{\nicefrac{1}{2}}(\R)}\leq(\lambda_1^{-1}+1)^\frac{1}{2}\|v\|_X.\]
Thus, \eqref{cont-emb} is proved.
\vskip2pt
\noindent
Now fix $u\in X$ and set $\tilde u=\alpha^{-\nicefrac{1}{2}}u$. By the results of Fiscella, Servadei $\&$ Valdinoci \cite{FSV}, we know that $C^\infty_c(0,1)$ is a dense linear subspace of $X$. So, there exists a sequence $(\psi_n)$ in $C^\infty_c(0,1)$ such that $\psi_n\to\tilde u$ in $X$. By \eqref{cont-emb}, we have $\psi_n\to\tilde u$ in $L_{\phi^*}(0,1)$ as well. In particular $\tilde u\in E_\phi$, namely the closure of the set of bounded functions of $X$ in $L_{\phi^*}(0,1)$. From a general result 
on Orlicz spaces (see \cite[formula (10.1), p. 81]{KR}) it follows that
\[\int_0^1\phi(\tilde u)dx<\infty,\]
which immediately yields the conclusion.
\end{proof}

\noindent
We conclude this section with a technical result which we shall use later:

\begin{lemma}\label{lions}
If $(v_n)$ is a sequence in $X$ with $\|v_n\|_X=1$ for all $n\in\N$ and $v_n\rightharpoonup v$ in $X$, $0<\|v\|_X<1$, then for all $0<\alpha<2\pi\omega$ and all $1<p<(1-\|v\|_X^2)^{-1}$ the sequence $(e^{\alpha v_n^2})$ is bounded in $L^p(0,1)$.
\end{lemma}
\begin{proof}
By applying the generalized H\"older inequality with exponents $\gamma_1,\gamma_2,\gamma_3>1$ such that $\gamma_1\alpha<2\pi\omega$ and
$\gamma_1^{-1}+\gamma_2^{-1}+\gamma_3^{-1}=1$, we have
\begin{align*}
\int_0^1 e^{p\alpha v_n^2}dx &=\int_0^1 e^{p\alpha [(v_n-v)^2+2(v_n-v)v+v^2 ]}dx \\
& \leq\left[\int_0^1 e^{\gamma_1 p\alpha(v_n-v)^2}dx\right]^\frac{1}{\gamma_1}\left[\int_0^1 e^{2\gamma_2 p\alpha(v_n-v)v}dx\right]^\frac{1}{\gamma_2}\left[\int_0^1 e^{\gamma_3 p\alpha v^2}dx\right]^\frac{1}{\gamma_3}.
\end{align*}
We estimate the three integrals separately. First we note that
\[\|v_n-v\|^2_X=1-2\langle v_n,v\rangle_X+\|v\|_X^2\to 1-\|v\|_X^2<\frac{1}{p},\]
so for $n\in\N$ big enough we have $\|v_n-v\|_X^2<\nicefrac{1}{p}$. Hence, by Corollary \ref{mti}
\[\int_0^1 e^{\gamma_1 p\alpha(v_n-v)^2}dx\leq\int_0^1 e^{\gamma_1\alpha\left(\frac{v_n-v}{\|v_n-v\|_X}\right)^2}dx\leq K_{\gamma_1\alpha}.\]
Besides, by Corollary \ref{mti} and Proposition \ref{sum} we have for some $c_1>0$
\begin{align*}
\int_0^1 e^{2\gamma_2 p\alpha(v_n-v)v}dx &\leq \int_0^1e^{2\left(\frac{\alpha}{2}\right)^{\nicefrac{1}{2}}\frac{v_n-v}{\|v_n-v\|_X}(c_1 v)}dx\leq\int_0^1e^{\frac{\alpha}{2}\left(\frac{v_n-v}{\|v_n-v\|_X}\right)^2+(c_1 v)^2}dx \\
& \leq\left[\int_0^1 e^{\alpha\left(\frac{v_n-v}{\|v_n-v\|_X}\right)^2}dx\right]^{\nicefrac{1}{2}}\left[\int_0^1 e^{2c_1^2 v^2}dx\right]^{\nicefrac{1}{2}}\leq K_\alpha\left[\int_0^1 e^{2c_1^2 v^2}dx\right]^{\nicefrac{1}{2}}.
\end{align*}
Finally, clearly
\[\int_0^1 e^{\gamma_3 p\alpha v^2}dx<\infty.\]
Thus, $(e^{\alpha v_n^2})$ is bounded in $L^p(0,1)$.
\end{proof}

\section{Proofs of Theorems~\ref{subcrit} and \ref{subcrit-sym}}

\noindent
In this section we act under ${\bf H}$. We give our problem a variational formulation by setting for all $u\in X$
\[\varphi(u)=\frac{\|u\|_X^2}{4\pi}-\int_0^1 F(u)dx.\]
Proposition \ref{sum}, ${\bf H}(i)$ and ${\bf H}(iv)$ imply that $\varphi\in C^1(X)$. By \eqref{equinorm}, its derivative is given for all $u,v\in X$ by
\begin{align*}
\langle\varphi'(u),v\rangle &=\frac{1}{2\pi}\langle u,v\rangle_X-\int_0^1 f(u)v dx \\
&= \int_\R(-\Delta)^{\nicefrac{1}{4}}u(-\Delta)^{\nicefrac{1}{4}}v dx-\int_0^1 f(u)v dx.
\end{align*}
In particular, if $u\in X$ and $\varphi'(u)=0$, then for all $v\in X$
\begin{equation}\label{weaks}
\int_\R(-\Delta)^{\nicefrac{1}{4}}u(-\Delta)^{\nicefrac{1}{4}}v dx=\int_0^1 f(u)v dx,
\end{equation}
namely $u$ is a (weak) solution of \eqref{problema-1d}.
\vskip2pt
\noindent
First we point out some consequences of ${\bf H}$. By ${\bf H}(iv)$, for all $\alpha>0$ there exists $c_2>0$ such that
\begin{equation}\label{grow}
|f(t)|\leq c_2 e^{\alpha t^2}, \,\,\quad \mbox{for all $t\in\R$.}
\end{equation}
By virtue of ${\bf H}(i)$, there exists $c_3>0$ such that
\begin{equation}\label{expo}
F(t)\geq c_3 e^\frac{|t|}{M}, \quad \mbox{for all $|t|\geq t_0$.}
\end{equation}
Finally, by ${\bf H}(i)$ and ${\bf H}(ii)$, for all $\varepsilon>0$ there exists $t_\varepsilon>0$ such that
\begin{equation}\label{ar}
F(t)\leq\varepsilon f(t)t, \,\,\quad\text{for all $|t|\geq t_\eps$}.
\end{equation}
\smallskip
\noindent
The following lemma shows a compactness property of $\varphi$:

\begin{lemma}\label{ps}
$\varphi$ satisfies the Palais-Smale condition at every level $c\in\R$.
\end{lemma}
\begin{proof}
Let $(u_n)$ be a sequence in $X$ such that $\varphi(u_n)\to c$ ($c\in\R$) and $\varphi'(u_n)\to 0$ in $X^*$. We need to show that $(u_n)$ has a convergent subsequence in $X$.
By (\ref{ar}), for all $0<\varepsilon<\nicefrac{1}{2}$ we can find $c_4>0$ such that for all $t\in\R$
\[F(t)\leq\varepsilon f(t)t+c_4.\]
For $n\in\N$ big enough we have $\varphi(u_n)\leq c+1$ and $\|\varphi'(u_n)\|_{X^*}\leq 1$, so
\begin{align*}
c+1 & \geq\frac{\|u_n\|_X^2}{4\pi}-\int_0^1[\varepsilon f(u_n)u_n+c_4]dx=\left(\frac{1}{2}-\varepsilon\right)\frac{\|u_n\|_X^2}{2\pi}+\varepsilon\langle\varphi'(u_n),u_n\rangle-c_4 \\
& \geq \left(\frac{1}{2}-\varepsilon\right)\frac{\|u_n\|_X^2}{2\pi}-\varepsilon\|u_n\|_X-c_4.
\end{align*}
Thus, $(u_n)$ is bounded in $X$. By Proposition \ref{poinc}, $(u_n)$ is bounded in $H^{\nicefrac{1}{2}}(\R)$ as well. By \cite[Theorem 7.1 and Theorem 6.10]{DPV},
passing to a subsequence we may assume that $u_n\rightharpoonup u$ in both $X$ 
and $H^{\nicefrac{1}{2}}(\R)$, and that $u_n\to u$ in $L^q(0,1)$ for all $q\geq 1$ and $u_n(x)\to u(x)$ a.e.\ in $(0,1)$. 
In particular, there exists $c_5>0$ such that $\|u_n\|_X^2\leq c_5$, for all $n\in\N$.
Observe that $(f(u_n))$ is bounded in $L^2(0,1)$. Indeed, by choosing
$0<\alpha<\pi\omega/c_5,$ by Corollary \ref{mti} and \eqref{grow} we get
\begin{equation}
\label{Lim-2}
\int_0^1 f^2(u_n)dx\leq c_2^2\int_0^1 e^{2\alpha u_n^2}dx\leq 
c_2^2\int_0^1 e^{2\alpha c_5 \left(\frac{u_n}{\|u_n\|_X}\right)^2}dx\leq c_2^2K_{2\alpha c_5}.
\end{equation}
Passing to a subsequence, we have $f(u_n)\rightharpoonup f(u)$ in $L^2(0,1)$. As a consequence, for all $v\in X$ we have
\[
\langle\varphi'(u),v\rangle=\frac{1}{2\pi}
\langle u,v\rangle_X-\int_0^1 f(u)v dx=\lim_n\langle\varphi'(u_n),v\rangle=0,
\]
namely $u$ is a solution of \eqref{problema-1d}. Observe that
\[
\lim_n\int_0^1 f(u_n) u_ndx=\int_0^1 f(u) u dx,
\]
since by \eqref{Lim-2} and $f(u_n)\rightharpoonup f(u)$ in $L^2(0,1)$ it holds
\begin{align}
\label{convfss}
&\Big|\int_0^1 f(u_n) u_n dx -\int_0^1 f(u) u dx\Big| \\
&\leq 
\|f(u_n)\|_{L^2(0,1)}\|u_n-u\|_{L^2(0,1)} 
+\Big|\int_0^1 (f(u_n)-f(u))u dx\Big|. \notag
\end{align}
In turn we have
\begin{equation*}
\lim_n \frac{\|u_n\|_X^2}{2\pi} =\lim_n\left[\int_0^1 f(u_n)u_n dx+\langle \varphi'(u_n),u_n\rangle\right]
=\int_0^1 f(u)u dx= \frac{\|u\|_X^2}{2\pi}, 
\end{equation*}
which immediately yields the assertion.
\end{proof}

\noindent
The following lemmas deal with the mountain pass geometry for $\varphi$:

\begin{lemma}\label{mount1}
There exist $\rho,a>0$ such that $\varphi(u)\geq a$ for all $u\in X$ with $\|u\|_X=\rho$.
\end{lemma}
\begin{proof}
By ${\bf H}(iii)$ there exist $0<\mu<\lambda_1$ and $\delta>0$ such that for all $|t|<\delta$ we have $F(t)\leq\mu t^2/(4\pi)$. Fix $q>2$, $0<\alpha<2\pi\omega$ and $r>1$ such that $r\alpha<2\pi\omega$ as well. By (\ref{grow}) there exists $c_6>0$ such that for all $|t|\geq\delta$ we have $F(t)\leq c_6 e^{\alpha t^2}|t|^q$. 
Summarizing, for all $t\in\R$, we obtain
\[F(t)\leq\frac{\mu t^2}{4\pi}+c_6 e^{\alpha t^2}|t|^q.\]
In what follows we use the estimate above, Proposition \ref{poinc}, Corollary \ref{mti} and the continuous embedding $X\hookrightarrow L^{r'q}(0,1)$. For all $u\in X$, $\|u\|_X\leq 1$ we have (for a convenient $c_7>0$)
\begin{align*}
\varphi(u) & \geq \frac{\|u\|_X^2}{4\pi}-\int_0^1\Big[\frac{\mu u^2}{4\pi}+c_6 e^{\alpha u^2}|u|^q \Big]dx \\
& \geq \left(1-\frac{\mu}{\lambda_1}\right)\frac{\|u\|_X^2}{4\pi}-c_6\Big(\int_0^1 e^{r\alpha u^2}dx\Big)^{1/r}\Big(\int_0^1 |u|^{r'q}\Big)^{1/r'} \\
& \geq \left(1-\frac{\mu}{\lambda_1}\right)\frac{\|u\|_X^2}{4\pi}-c_7 \|u\|_X^q.
\end{align*}
Set for all $t\geq 0$
\[g(t)=\left(1-\frac{\mu}{\lambda_1}\right)\frac{t^2}{4\pi}-c_7 t^q.\]
By a straightforward computation we find $0<\rho<1$ such that $g(\rho)=a>0$. So, for all $u\in X$ with $\|u\|_X=\rho$ we have $\varphi(u)\geq a$.
\end{proof}

\begin{lemma}\label{asympt}
If $Y\subset X$ is a linear subspace generated by bounded functions and ${\rm dim}(Y)<\infty$, then $\sup_{u\in Y}\varphi(u)<\infty$ and 
$$
\lim_{\substack{\|u\|_X\to\infty \\ u\in Y}}\varphi(u)=-\infty.
$$
\end{lemma}
\begin{proof}
Fix $p>2$. By (\ref{expo}), we have $|t|^{-p} F(t)\to\infty$ for $|t|\to\infty$,
so we can find $c_{8}>0$ such that for all $t\in\R$ we have $F(t)\geq |t|^p-c_{8}$. 
Whence, for some $c_9>0$, we obtain for all $u\in Y$
\begin{equation*}
\varphi(u)  \leq  \frac{\|u\|_X^2}{4\pi}-\|u\|_{L^p(0,1)}^p+c_{8}
\leq \frac{\|u\|_X^2}{4\pi}-c_9\|u\|_X^p+c_{8},
\end{equation*}
which readily yields the assertion.
\end{proof}

\noindent
{\em Proof of Theorem~\ref{subcrit} concluded.} The existence of one solution follows
by applying the Mountain Pass Theorem (see Rabinowitz \cite[Theorem 2.2]{rab-book}) to $\varphi$ and combining Lemmas \ref{ps}, \ref{mount1} and \ref{asympt}.
\vskip2pt
\noindent
Concerning the multiplicity, we apply \cite[Theorem 9.12]{rab-book}. \qed

\vskip5pt
\noindent
{\em Proof of Theorem~\ref{subcrit-sym} concluded.} 
Given a nonnegative function $u\in X$ and any $H=(a,\infty)$ with $a<0$, 
we have the following inequality for the polarization $u^H$ (see Baernstein \cite[Theorem 2, p. 58]{baer})
\[
\int_{\R^2}\frac{(u^H(x)-u^H(y))^2}{|x-y|^2}dxdy
\leq \int_{\R^2}\frac{(u(x)-u(y))^2}{|x-y|^2}dxdy,
\]
which implies that $\varphi(u^H)\leq \varphi(u)$, for all nonnegative $u$ of $X$.
\vskip2pt
\noindent
The existence of an even solution on $(-1,1)$, decreasing on $(0,1)$, equal to zero on $\R\setminus(-1,1)$ follows
by the (symmetric) Mountain Pass Theorem of Van Schaftingen
\cite[Theorem 3.2]{vanSch} applied to the functional $\varphi$ on $X$ with the $V$ therein
chosen as $V=L^2(-1,1)$, and on account of Lemmas \ref{ps}, 
\ref{mount1} and \ref{asympt}. \qed

\begin{example}\label{ex1}\rm
Fix $1<q<2$ and $0<\mu<\nicefrac{\lambda_1}{2\pi}$. 
Define $f:\R\to\R$ by setting, for all $t\geq 0$,
\[f(t)=\left\{
	\begin{array}{ll}
		\mu t  & \mbox{if } 0\leq t\leq 1, \\
		\mu t^{q-1} e^{t^q-1} & \mbox{if } t>1,
	\end{array}
\right.
\]
and $f(t)=-f(-t)$ for all $t<0$. It is easily seen that $f$ is continuous, odd and satisfies ${\bf H}$. By Theorem \ref{subcrit}, then, the corresponding problem \eqref{problema-1d} admits infinitely many solutions.
\end{example}

\section{Proof of Theorem~\ref{C-state}}

\noindent
In this section, we consider the critical case, that is, we act under ${\bf H'}$.
\vskip2pt
\noindent
An important remark here is that (\ref{grow}) holds only for $\alpha>\alpha_0$. We prove that the Palais-Smale condition is satisfied only for levels in a certain range:

\begin{lemma}\label{ps-2}
If $f$ satisfies ${\bf H'}$, then $\varphi$ satisfies the Palais-Smale condition at any level $c<\omega/(2\alpha_0)$.
\end{lemma}
\begin{proof}
Let $(u_n)$ be a sequence in $X$ such that $\varphi(u_n)\to c$ and $\varphi'(u_n)\to 0$ in $X^*$. Arguing as in the proof of Lemma~\ref{ps}, it is readily seen that there exists a positive
constant $c_{10}$ such that, for all $n\in\N$,
\[\max\Big\{\|u_n\|_X^2,\int_0^1 f(u_n)u_n dx,\int_0^1 F(u_n)dx\Big\}\leq c_{10}.\]
Moreover, up to a subsequence, $u_n\rightharpoonup u$ in $X$ and $u_n\to u$ in $L^q(0,1)$ for all $q\geq 1$. Reasoning as in \cite[Lemma 2.1]{DMR} we have $f(u_n)\to f(u)$ in $L^1(0,1)$.
Whence, in light of ${\bf H}(i)$ it follows that 
$\int_0^1 F(u_n)dx\to\int_0^1 F(u)dx$. So we have
\begin{equation}\label{norm-conv}
\frac{\|u_n\|_X^2}{4\pi}\to c+\int_0^1 F(u)dx.
\end{equation}
Then, since $\varphi'(u_n)\to 0$, we get
\[\int_0^1 f(u_n)u_n dx\to 2\Big(c+\int_0^1 F(u)dx\Big).\]
So, by means of ${\bf H}(ii)$, we have
\[c=\frac{1}{2}\lim_n\int_0^1\left[f(u_n)u_n-2F(u_n)\right]dx\geq 0.\]
Besides, for all $v\in C^\infty_c(0,1)$ we have
\[
\langle\varphi'(u),v\rangle=\frac{1}{2\pi}
\langle u,v\rangle_X-\int_0^1 f(u)v dx=\lim_n\langle\varphi'(u_n),v\rangle=0.\]
Recalling again the density result of \cite{FSV}, we have $\langle\varphi'(u),v\rangle=0$ for all $v\in X$, namely $u$ is a solution of \eqref{problema-1d}. By ${\bf H}(ii)$ and taking $v=u$ we have
\[
\varphi(u)=\frac{1}{2}\Big(\frac{\|u\|_X^2}{2\pi}-2\int_0^1 F(u)dx\Big)\geq\frac{1}{2}\Big(\frac{\|u\|_X^2}{2\pi}-\int_0^1 f(u)udx\Big)=0.
\]
Summarizing, we have $c\geq 0$ and $\varphi(u)\geq 0$. Now we distinguish three cases.
\begin{itemize}[leftmargin=.2in]
\item[$(a)$] If $c=0$, then by virtue of \eqref{norm-conv} and $\varphi(u)\geq 0$, we get
\[
\frac{\|u\|_X^2}{4\pi}\geq\int_0^1 F(u)dx=\lim_n\frac{\|u_n\|_X^2}{4\pi}.
\]
Recalling that $u_n\rightharpoonup u$ in $X$, we conclude that $u_n\to u$ in $X$.
\item[$(b)$] If $c>0$, $u=0$, then the sequence $(f(u_n))$ is bounded in 
$L^q(0,1),$ for some $q>1$. Indeed, since $c<\omega/(2\alpha_0)$ we can find $q>1$, $\eps>0$ and $\alpha_0<\alpha<2\pi\omega$ such that $2\pi (2c+\eps)q\alpha:=\beta<2\pi\omega$. Since $\|u_n\|_X^2\to 4\pi c$, for $n\in\N$ big enough we have $\|u_n\|_X^2<2\pi(2c+\eps)$. So, applying \eqref{grow} and Corollary \ref{mti} we have
\[
\int_0^1|f(u_n)|^qdx\leq c_2^q\int_0^1 e^{q\alpha u_n^2}dx\leq c_2^q\int_0^1 e^{\beta\big(\frac{u_n}{\|u_n\|_X}\big)^2}dx\leq c_2^q K_\beta.
\]
\noindent
Recalling that $u_n\to 0$ in $L^{q'}(0,1)$ and that
\[
0\leq\int_0^1 f(u_n)u_n dx\leq\|f(u_n)\|_{L^q(0,1)}\|u_n\|_{L^{q'}(0,1)},
\]
from $\varphi'(u_n)\to 0$ we have immediately
\[\lim_n\frac{\|u_n\|_X^2}{2\pi}=\lim_n\int_0^1 f(u_n)u_ndx=0,\]
whence $u_n\to 0$ in $X$. Thus $\varphi(u_n)\to 0<c$, a contradiction.
\item[$(c)$] If $c>0$, $u\neq 0$, then we prove that $\varphi(u)=c$. 
This equality yields the strong convergence by means of \eqref{norm-conv}.
We know that $\varphi(u)\leq c$, so by contradiction assume $\varphi(u)<c$. Then 
\[
\|u_n\|_X^2\to 4\pi\Big(c+\int_0^1 F(u)dx\Big)>\|u\|_X^2.
\]
Set $v_n=\|u_n\|_X^{-1}u_n$ and 
$v=\big(4\pi c+4\pi\int_0^1 F(u)dx\big)^{-\nicefrac{1}{2}}u$. So we have 
$\|v_n\|_X=1$, $0<\|v\|_X<1$ and $v_n\rightharpoonup v$ in $X$. Since $c<\omega/(2\alpha_0)$, we can find $q>1$, $\alpha_0<\alpha<2\pi\omega$ such that $qc<\omega/(2\alpha)$, hence (recall $\varphi(u)\geq 0$)
\[2q\alpha<\frac{\omega}{c-\varphi(u)}.\]
We have
\[\lim_n q\alpha\|u_n\|_X^2=4\pi q\alpha\Big(c+\int_0^1 F(u)dx\Big)<2\pi\omega\frac{c+\int_0^1 F(u)dx}{c-\varphi(u)}.\]
We can choose $p>1$, $0<\gamma<2\pi\omega$ such that
\[p<\frac{c+\int_0^1 F(u)dx}{c-\varphi(u)}=\frac{1}{1-\|v\|_X^2}\]
and for $n\in\N$ big enough
\[q\alpha\|u_n\|_X^2<p\gamma.\]
Since $\gamma<2\pi\omega$, by Lemma \ref{lions} the sequence $(e^{\gamma v_n^2})$ is bounded in $L^p(0,1)$, so
\[
\int_0^1|f(u_n)|^q dx\leq c_2^q\int_0^1 e^{q\alpha u_n^2}dx
\leq c_2^q\int_0^1 e^{p(\gamma v_n^2)}dx,
\]
which proves that $(f(u_n))$ is bounded in $L^q(0,1)$. Passing if necessary to a subsequence, we have $f(u_n)\rightharpoonup f(u)$ in $L^q(0,1)$ while $u_n\to u$ in $L^{q'}(0,1)$. So,
\begin{align*}
\label{convfss}
&\Big|\int_0^1 f(u_n) u_n dx -\int_0^1 f(u) u dx\Big| \\
&\leq 
\|f(u_n)\|_{L^q(0,1)}\|u_n-u\|_{L^{q'}(0,1)} 
+\Big|\int_0^1 (f(u_n)-f(u))u dx\Big|, \notag
\end{align*}
hence
\[\lim_n\int_0^1 f(u_n)u_n dx=\int_0^1 f(u)udx.\]
As above, this yields $u_n\to u$ in $X$. This in turn implies $\varphi(u)=c$, a contradiction. 
\end{itemize}
This concludes the proof.
\end{proof}

\medskip
\noindent
{\em Proof of Theorem~\ref{C-state} concluded.}
The conclusions of Lemmas \ref{mount1} and \ref{asympt} still hold, with small changes in the proofs. Moreover, if we fix $t>0$ such that $\varphi(t\psi)<\varphi(0)$ and denote by $\Gamma$ the set of continuous paths in $X$ joining $0$ and $t\psi$ and set
\[c=\inf_{\gamma\in\Gamma}\max_{\tau\in[0,1]}\varphi(\gamma(\tau)),\]
by ${\bf H'}(v)$ we see that $c<\omega/(2\alpha_0)$. Thus, by Lemma \ref{ps-2}, $\varphi$ satisfies the Palais-Smale condition at level $c$. By the Mountain Pass Theorem, then, \eqref{problema-1d} has a nontrivial solution.
\qed

\begin{example}\label{ex2}\rm
Fix $0<\mu<\nicefrac{\lambda_1}{2\pi}$, $0<\alpha_0<2\pi\omega$. Define $f:\R\to\R$ by setting, for all $t\geq 0$,
\[f(t)=\left\{
	\begin{array}{ll}
		\mu t  & \mbox{if } 0\leq t\leq 1, \\
		\mu te^{\alpha_0(t^2-1)} & \mbox{if } t>1,
	\end{array}
\right.
\]
and $f(t)=-f(-t)$ for all $t<0$. It is easily seen that $f$ is continuous and satisfies ${\bf H'}(i)-(iv)$. If there exists $\psi\in X$ satisfying ${\bf H'}(v)$, then by Theorem \ref{C-state} the corresponding problem \eqref{problema-1d} admits a nontrivial solution.
\end{example}

\end{document}